\documentclass[10pt]{article}
\usepackage{amsmath,amssymb,amsfonts,amsthm,graphicx}

\title{Generating Functions for Special Flows over the 1-Step Countable Topological Markov Chains }
\author{ D. Ahmadi Dastjerdi
\\ %\thanks{ }
 S. Lamei
\\
\footnotesize{Department of Mathematics, The University of
Guilan, P.O. 1914, Rasht, Iran.
} \\ \footnotesize{e-mail: ahmadi@guilan.ac.ir}\\
\footnotesize{Department of Mathematics,} \footnotesize{The
University
of Guilan, P.O. 1914, Rasht, Iran. }  \\
\footnotesize{e-mail:lamei@guilan.ac.ir  }}

\date{}

\begin{document}
 \maketitle \thispagestyle{empty}
\newtheorem{The}{Theorem}
\newtheorem{lem}{Lemma}
\newtheorem{prop}{Proposition}
\newtheorem{cor}{Corollary}
\newtheorem{hyp}{Hypothesis}
\newtheorem{rem}{Remark}

\pagestyle{headings} \markright{}

\begin{abstract}
Let $Y$ be a topological Markov chain with finite leading and
follower sets. Special   flow over $Y$ whose  height function
depends on the time zero of elements of $Y$ is constructed. Then a
formula for computing the entropy of this flow will be given. As
an application,  we give a lower estimate for the entropy of a
class of geodesic flows on the modular surface. We also give
sufficient conditions to guarantee the existence of a measure with
maximal entropy. \\
\end{abstract}
\maketitle
\section{Introduction}\label{intro}
There are two main routines to compute the entropy of non-compact
dynamical systems. The first is to use $(T,\epsilon)$-spanning
sets introduced by Bowen  on  metric of spaces
\cite{Bow,PP,Thompson}, and the second is to use the topological
pressure from the thermodynamic formalism
\cite{BaIo,Kempton,Marc,S}. Our concern is the latter and in
particular, we consider  special flows over countable Markov
chains. These flows are mainly associated with geodesic flows on
non-compact manifolds with negative curvature. For instance in
\cite{GK}, it has been shown that the geodesic flows on the
modular surface can be represented by special flow over countable
alphabet. However, even in this case, depending on the
properties, several definitions for entropies are given
\cite{BaIo,Kempton,Marc,S}.

In this paper we construct  a special flow constructed from a
certain class of topological Markov chains. Namely, we let the
base $Y$ of the flow be taken from a 1-step topological Bernoulli
scheme TBS with countable states so that the follower and leading
sets are finite. This means that we partition the set of alphabet
to $\{ P_1, ... , P_m\}$ such that if $y=\{y_i \}_{i \in
\mathbb{Z}}$, $y'=\{y_i' \}_{i \in \mathbb{Z}}$ are in $Y$ with
$y_0$, $y_0' \in P_k$ then $y_1, y_1' \in P_i$ and $y_{-1},
y_{-1}' \in P_j$ for some $1\leq i, j , k \leq m$. We call any
$Y$ satisfying this condition \textit{reducible to finite type }
or briefly  RFT.

%This work can be considered as an extension of \cite{P} where
%studies the \textit{local perturbation} of a TBS which is
%obtained by removal of finite edges in the corresponding graph.
%Our goal  is to find the topological entropy of the foresaid
%flows based on the results in \cite{P} and \cite{S}.
We let $Y$ to be a 1-step topological Markov chain where this will
able us to define the height function depending only on the zero
coordinate of an element of $y$ where $y=\{y_i\}_{i \in
\mathbb{Z}}$. Our main objective is to find the generating
function $\phi: \mathbb{R} \rightarrow \mathbb{R}$ depending on
the height function. Then by applying some conditions, necessary
for the results in \cite{S}, the topological entropy of the flow
is $-\ln(x_0)$ where $x_0$ is the unique point where $\phi(x_0)
=1$. In section \ref{egs}, we show how this is applied in
application, by giving some examples which arises in the study of
geodesic flows in the modular surface. Then in Section \ref{4},
we show that results in \cite{P} can be deduced from our method.
%by introducing some new
%notations and adopting some from \cite{P}, we show that
%restricting our conditions to the local perturbation considered
%in \cite{P}, gives the same generating function as one expects.
In Section \ref{crit}, based on results in \cite{P} and \cite{S},
 sufficient conditions  for having a measure with maximal entropy for the flow
has been given.

\textit{Acknowledgments.} We thank S. Savchenko for bringing to
our attention a series which helped us to give the example 3 in
section 3. Also we thank M. Kessebohmer for his useful comments.

\section{Notations and main definitions} Now we recall some notations
and definitions many adopted from \cite{P}. Let $G$ be a
connected directed graph with a countable vertex set $V(G)$ and
edge set $E(G) \subseteq V(G) \times V(G)$. If $E(G) = V(G)
\times V(G)$, then we denote $G$ by $G_0$ which is called a
\textit{complete graph}. A path $\gamma$ with length
$\ell(\gamma)=n$ in $G$ from $v_0$ to $v_n$ is the sequence
$\gamma=(v_0...,v_n)$, $n\geq 1$ of vertices of $G$ such that
$(v_k,v_{k+1})\in E(G)$ for $0\leq k \leq n-1$. A path
$\gamma=(v_0,...,v_n)$, $n\geq 1$ in $G$ is called a
\textit{simple $v$-cycle} if $v_0=v_n=v$ and $v_i \neq v$ for
$1\leq i \leq n-1$. Denote by $C(G;v)$ the set of all simple
$v$-cycles in the graph $G$.

Let $Y(G)=\{(...,y_{i-1}, y_i, y_{i+1}, ...) :\ y_i \in V(G), \
(y_i,y_{i+1})\in E(G), i \in \mathbb{Z} \}$ be the set of
two-sided infinite paths in $G$ and the shift transformation
$T:Y(G) \longrightarrow Y(G) $ is defined as $(Ty)_i= y_{i+1}$,
for $y\in Y(G)$ and $i\in \mathbb{Z}$.  The system ($Y(G), T$) is
called a \textit{countable Topological Markov Chain} TMC. In the
case of complete graph, the dynamical system ($Y(G_0), T$) is
called \textit{countable Topological Bernoulli Scheme} TBS. A TMC
will be a \textit{local perturbation of a countable} TBS if
$D=E(G_0)-E(G)$ is finite.

Let $(Y(G),T)$ be a given TMC. Consider on $Y(G)$ a continuous
positive function $f :Y(G) \longrightarrow (0,\infty)$ such that
$\sum_{k=1}^{\infty}f(T^{k}y)
=\sum_{k=1}^{\infty}f(T^{-k}y)=\infty $, $y\in Y(G)$. The set of
such functions which depend only on zero coordinate $y_0$ of the
sequence $y$ is denoted by $\mathcal{F}^o(Y(G))$. A good
candidate for such $f$ is a special flow. To be more precise, let
$Y_f(G)=\{ (y,u): \ y\in Y(G), \ 0\leq u \leq f(y) \}$ with the
points $(y,f(y))$ and $(Ty,0)$ identified. For $0\leq u, u+t \leq
f(y)$ we let $T^t_f(y,t)=(y,u+t)$. The family $T_f =\{ T^t_f \}$,
$t\in \mathbb{R} $ is a \textit{special flow} constructed over
its base $Y(G)$.

%The generating functions are our main tool for computing entropy
%?? of the special flow $T_f$ obtained from a TBS $(Y(G_0),T)$ is
%defined to be the series
Consider the following series
\begin{eqnarray}\label{F}
F_{f,V}(x)=\sum_{v\in V}x^{f(v)},
 \end{eqnarray}
 which is
defined  for $x\geq 0.$ This series is convergent at zero and we
will call $r(F_{f,V})=\sup \{ x\geq 0: \sum_{v\in V} x^{f(v)} \
\textrm{is convergent} \}$ the \textit{radius of convergent} of
$F_{f,V}.$

 Let $f \in \mathcal{F}^o(Y(G))$.  The \textit{generating function} of simple $v$-cycles with
respect to the special flow $T_f$ constructed over a TMC $(Y(G),
T)$  is defined to be the series
\begin{eqnarray}\label{phi1}
\phi_{G,f,w}(x)= \sum_{\gamma\in C(G;w)}x^{f^*(\gamma)},
\hspace{1cm} x\geq 0
\end{eqnarray}
where $f^*(\gamma)=\sum_{i=0}^{n-1}f(v_i),
\gamma=(v_0,...,v_{n})$. The radius of convergence
$r(\phi_{G,f,w})$ $\in[0,1)$ is defined as
was defined for (\ref{F}).\\
 \section{Computing Generating Function }\label{generating}
Let $V(G)$ and $E(G)$ be as above. Consider a ``weighted''
adjacent matrix $A_{G}=[a_{ij}]$, that is, a matrix where
$a_{ij}=x^{f(v_i)}$ if $(v_i,v_j)\in E(G)$ and zero otherwise.
 For each $v \in V(G)$, let $V_v^+=\{v'
\in V(G): (v,v') \in E(G) \}$ and $V_v^-=\{v' \in V(G): (v',v) \in
E(G) \}$ be the \textit{follower} and \textit{leading} set for $v$
respectively.
 Set $\widetilde{V}=\{v \in V(G):
\exists v'\in V(G)\ni (v,v')\not \in E(G)\ \mathrm{or}\ (v',v)
\not \in E(G) \}$. Note that $\widetilde{V}$ may be infinite or
even equal to $V(G)$. Let $\rho$ be an equivalence relation on
$V(G)$ defined by $v\stackrel{\rho}{\sim}v' \Leftrightarrow
(V_v^+=V_{v'}^+ , V_v^- =V_{v'}^- )$ and $P$ be the associated
partition. We are interested in cases where $|P| < \infty$.

We alter a bit the above notations and will produce a quotient set
for $G$ which is again a connected directed graph. To achieve that
fix $w \in V(G)$ and let $W_{\{w\}} = P \vee \{ \{w\} ,
V(G)-\{w\} \} $ be the set of all non-empty intersections of $P$
with the partition $\{\{w\} , V(G)-\{w\}\}$. Then
$|W_{\{w\}}|<\infty$ and let $V_0 =\{w\}, V_1, \ldots , V_m$ be
the elements of $W_{\{w\}}$.

For $v \in V_i$ define the follower and leading sets for $V_i$ as
$V_{i}^+ = V_v^+$ and $V_{i}^- = V_v^-$ respectively. Note that
any $V_i$, $V_{i}^+$ or $V_{i}^-$ can be written as the union of
some of elements of $W_{\{w\}}$. Therefore, a directed graph $H$
arises with vertex set $W_{\{w\}}$ and the edge set $E(H)= \{
(V_i,V_j): (V_i,V_j) \in W_{\{w\}}\times W_{\{w\}}, (v_i,v_j) \in
E(G), v_i \in V_i, v_j \in V_j \}$. We call $H$ the
\textit{quotient graph} for $G$. Graph $H$ is connected because
$G$ is connected.

%p2____________________
%The following results (up to Lemma \ref{ours}) will not be used
%for our later task and maybe skipped. However, we bring it here
%to clarify our idea.

%Let $H_V$ be a tree with two levels from $H$: consisting of root
%$V$ and all the following sets of $V$ as the second level. Now
%consider the tree $H_{\{w\}}$.
 By reindexing the elements of
$W_{\{w\}}$, we may assume $\{ V_1, \ldots,V_k \} = V_{\{w\}}^+$.
Let $H_{\{w\}}$ be a tree with root $\{w\}$ and $V_1$, ..., $V_k$
in its second level.
 We
 wish to extend $H_{\{w\}}$ to a  tree $T_{\{w\}}$ whose first two levels
are exactly $H_{\{w\}}$ and any path starting from $\{w\}$ ends
at $\{w\}$. By this we mean the third level of $T_{\{w\}}$
consists of the follower sets of $V_j$'s, $V_j \neq \{w\}$ and
$1\leq j \leq k$. Again the next level consists of the follower
sets of vertices of third level which are not $\{w\}$ and so on.
\begin{The}
Let $k$ be the number of vertices at the second level of
$T_{\{w\}}$. Then all the elements of $W_{\{w\}}$ appear at the
vertices of $T_{\{w\}}$ at most up to level $m-k+2$.
\end{The}
\begin{proof}
If $k=m$ we are done. So assume $k<m$. Then the third level must
have a vertex which is not in $\{V_1,\ldots,V_m\}$. Otherwise,
that vertex will not appear in any level which is in
contradiction with the fact that $H$ is connected. By the same
reasoning, any higher level must have one new vertex until all of
them have appeared.
\end{proof}
%p3___________________________________

This theorem justifies that such $T_{\{w\}}$ exists. Because by
replacing $\{w\}$ with any other vertex $V$ in $W_{\{w\}}$ and
using the same proof as the above theorem, $\{w\}$ will appear at
least once as vertex in a tree with root at $V$. %We continue this
%process until $T_{\{w\}}$ is constructed. %\newpage

The next lemma states that how in our case the computation of
generating function can be simplified.  Let
$\alpha_i=\alpha_i(x)=\sum_{v\in V_i}x^{f(v)}$ and set
$\alpha_{ij}=\alpha_{ij}(x)=\alpha_i(x)$ if $(V_i,V_j) \in E(H)$
and zero otherwise.
\begin{lem}\label{ours} Suppose  ($Y(G),T$) is an RFT and $f
\in \mathcal{F}^0(Y(G_0))$ with $r(F_{f,V})>0$. Then there exist
series $A_i(x)$  which are the solution of the follower set of
equations
\begin{eqnarray}\label{A}A_i(x)=\alpha_{i0}(x)+\alpha_{i1}(x)A_{1}(x) +
\alpha_{i2}(x)A_{2}(x)+ ... + \alpha_{im}(x)A_{m}(x),
\end{eqnarray} for $1\leq i\leq m$ so that the generating
function for the flow $T_f$ is
\begin{eqnarray}\label{phi}\phi_{G,f,w}(x)=\alpha_{00}(x)+\alpha_{01}(x)A_{1}(x)
+\alpha_{02}(x)A_{2}(x) + ...
+\alpha_{0m}(x)A_{m}(x).
\end{eqnarray}
Here $r(\phi_{G,f,w}) = \min \{ r(A_1), ..., r(A_m) \}\leq
r(F_{f,V}).$
\end{lem}
\begin{proof}
Let $$A_i(x)=\sum_{v\in V_i}
\sum_{\gamma=(v,...,w)}x^{f^*(\gamma)}$$ be a series on all paths
in $G$ starting at a vertex $v\in V_i$ and ending at $w$. Then
\begin{eqnarray*}
A_i(x)&=& \left(\sum_{v\in V_i}x^{f(v)}\right)\sum_{V_j \in
V_i^+}\sum_{v'\in V_j}\sum_{\gamma'
=(v',...,w)}x^{f^*(\gamma')}\\
&=&\alpha_{i0}(x)+  \alpha_{i1}(x)A_{1}(x) +
\alpha_{i2}(x)A_{2}(x)+ ... + \alpha_{im}(x)A_{m}(x).
\end{eqnarray*}
Since $V_i \cap V_j = \emptyset$ for $i \neq j$,
$F_{f,V}(x)=\sum_{i=1}^{m}\alpha_i(x).$ Hence $$r(F_{f,V})= \min\{
r(\alpha_i), ..., r(\alpha_m) \}$$ and since each $A_i(x)$ is a
rational map in variables $\alpha_1(x), ..., \alpha_m(x)$,
therefore $\min\{ r(A_1), ..., r(A_m) \}\leq r(F_{f,V}).$ Also
\begin{eqnarray*}
\phi_{G,f,w}(x)&=&\sum_{\gamma \in C(G;w)}x^{f^*(\gamma)}\\
&=&x^{f(w)} \sum_{V_i \in V^+_{w}}\sum_{v\in V_i}
\sum_{\gamma=(v,...,w)}x^{f^*(\gamma)}\\
&=&\alpha_{00}(x)+  \alpha_{01}(x)A_{1}(x)
+\alpha_{02}(x)A_{2}(x) + ... +\alpha_{0m}(x)A_{m}(x).
\end{eqnarray*}
By the way $A_i(x)$ is defined above, $r(\phi_{G,f,w})=
\min\{r(A_1(x)), ... , r(A_m(x)) \}. $
\end{proof}

Set $A(x)=(A_1(x),...,A_m(x))$ and $\alpha(x)=(\alpha_1(x), ... ,
\alpha_m(x))$ and consider them as column vectors when it applies
and let
\begin{eqnarray}\label{M} M(x)=
  \begin{pmatrix}
    \alpha_{11}(x)-1

     & \alpha_{12}(x) & ... & \alpha_{1m}(x) \\
    \alpha_{21}(x) & \alpha_{22}(x)-1 & ... & \alpha_{2m}(x) \\
    \vdots & & & \\
    \alpha_{m1}(x) & \alpha_{21}(x) & ... & \alpha_{mm}(x)-1
  \end{pmatrix}.
\end{eqnarray}
 Then statement (\ref{A}) in the conclusion of
Lemma (\ref{ours}) implies
 \begin{eqnarray}\label{5} M(x)\,
A(x) =-\alpha(x).
\end{eqnarray}
Consider (\ref{5}) as a set of equations with unknown $A(x)$.  In
the next theorem, we will find  $x$ such that $A(x)$ satisfies
(\ref{5}) and $A(x)$ is a solution of (\ref{A}), that is, we will
find $r({\phi_{G,f,w}})$. In fact for $x>0$, a solution of
(\ref{5}) is a solution of (\ref{A}) if and only if $A_i(x)>0$.

Set $\tilde{x_0}=r(F_{f,V})$ if $M(x)$ is invertible for $0\leq
x<r(F_{f,V})$, otherwise set $\tilde{x_0}=\inf \{x:\ 0\leq
x<r(F_{f,V}),\ \det M(x)=0 \}$. Since $M(0)$ is invertible and
$M(x)$ is continuous then clearly $\tilde{x_0}>0$ if and only if
$r(F_{f,V})>0$.
%But first if there is $x$ such that $\det M(x)=0$. Then set
%$\tilde{x_0}=\inf \{x:\ 0\leq x<r_{F_{f,V}},\ \det M(x)=0 \}$,
%otherwise set $\tilde{x_0}=r(F_{f,V})$.
\begin{The}\label{det}
Suppose the hypothesis of Lemma (\ref{ours}) is satisfied and
$A(x)$, $M(x)$ and $\tilde{x_0}$ are as above. Then
$r(\phi_{G,f,w})=\tilde{x_0}$ and if $\tilde{x_0}< r(F_{f,V})$,
$\lim_{x\rightarrow \tilde{x_0}}A_i(x)=\lim_{x\rightarrow
\tilde{x_0}}\phi_{G,f,w}(x)=\infty$.
\end{The}
%If for some $0<x<r(F_{f,V})$, $M(x)$ is nonsingular, we set
%$\tilde{x_0}$ to be the smallest of such $x$'s.
%\begin{The}
%Let $A(x)$, $M(x)$ and $\tilde{x_0}$ be as above. If for some
%$0<x<r(F_{f,V})$, $\det M(x)=0$, then
%$r(\phi_{G,f,w})=\tilde{x_0}$ and $\lim_{x\rightarrow
%\tilde{x_0}}A_i(x)=\lim_{x\rightarrow
%\tilde{x_0}}\phi_{G,f,w}(x)=\infty$. Otherwise,
%$r(\phi_{G,f,w})=r(F_{f,V}).$
%\end{The}
\begin{proof}
 Let $e_i=(0,0,...,1,0,...,0)$ be the unit vector whose
$i$th entry is 1. Let ${\cal P}=\{ \sum_{i=1}^{m}t_ie_i: \ t_i\geq 0
\}$ and ${\cal N}=\{  \sum_{i=1}^{m}t_ie_i: \ t_i\leq 0 \}$. The
boundary of  ${\cal P}$ consists of $m$ sets $P_j=\{ \sum_{i=1,
i\neq j}^m t_ie_i:\ t_i\geq 0 \} $, $1\leq j\leq m$. These are $m-1$
dimensional manifolds with boundaries.  Let $\mathcal{M}_x:
v(x)\mapsto M(x)v(x)$, $\mathcal{S}^+(x)=
\mathcal{M}_x(\mathcal{P})$ and $S_j(x)= \mathcal{M}_x (P_j)$. Then
$\partial \mathcal{P}=\cup_{j=1}^{m}P_j$ and if $M(x)$ is
invertible, $\partial \mathcal{S}^+= \cup_{j=1}^m S_j(x)$.

%If $M(x)$ is invertible then $\mathcal{S}(x)=\sum_{i=1, t_i\geq
%0}^m t_i (M(x)e_i)$ the image of $\mathcal{P}$ under the linear
%map  is $m$ dimensional. Let  $S_j= M(x)P_j$ and $S^+=
%\cup_{j=1}^m S_j$ be the boundaries of $\mathcal{P}$ and
%$\mathcal{S}(x)$ respectively.

Recall that if $W$ is a subspace of $\mathbb{R}^m$ of codimension
1, then $\mathbb{R}^m\backslash W$ consists of two unbounded
components. But when $M(x)$ is invertible,
$\partial\mathcal{S}^+(x)$ is a homeomorphic image of such a $W$
and hence $\mathbb{R}^m \backslash
\partial\mathcal{S}^+(x)$ consists of two unbounded components as well. Also for
any $x$ and $j$ we have $S_j(x)\cap \mathcal{N}^o =\{0\}$ where
$\mathcal{N}^o$ is the interior of $\mathcal{N}$. That is because
 the $j$th entry of a nonzero vector in $S_j(x)$ which is equal to $t_1\alpha_{1j}(x)+
...+t_{j-1}\alpha_{(j-1)j}(x)+t_{j+1}\alpha_{(j+1)j}(x)+...+
t_{m}\alpha_{mj}(x)$ for nonzero $t_i$'s is never negative. Hence
as far as $M(x)$ is invertible, $\mathcal{N}$ lies in one side of
$\mathbb{R}^m \backslash
\partial\mathcal{S}^+(x)$.

Note that $M(0)=-Id$ and hence it is invertible and
$\mathcal{P}=\mathcal{M}^{-1}_0(\mathcal{N})$. Also by continuity,
for small positive $x$, $M(x)$ remains invertible and in fact
$\mathcal{P}\subset \mathcal{M}^{-1}_x(\mathcal{N})$. It may happen
that we have $x$ such that $M(x)$ is not invertible and set as above
$\tilde{x_0}$ to be the smallest positive real such that $\det
M(\tilde{x_0})=0$. Since the entries of  $M(x)$ does not decrease as
$x$ increases, so $\mathcal{S}^+(x_1)\subseteq \mathcal{S}^+(x_2)$
when $x_1 < x_2 < \tilde{x_0}$. Hence $\mathcal{M}^{-1}_{x_2}
\mathcal{S}^+(x_1) \subseteq
\mathcal{M}^{-1}_{x_2}\mathcal{S}^+(x_2)$ and then
$\mathcal{M}^{-1}_x(\mathcal{S}^+(0))=
\mathcal{M}^{-1}_x(\mathcal{N}) \subseteq \mathcal{P}$ for $0<x<
\tilde{x_0}$.

In particular, $A(x)=\mathcal{M}^{-1}_{x}(-\alpha) \in
\mathcal{P}^o$ where by uniqueness of solutions $A(x)$ will be
the same as (\ref{A}).  We are done if we show that for each $i$
\begin{equation}\label{6}
  \lim_{x\rightarrow \tilde{x_0}}A_i(x)=\infty.
\end{equation}
First note that since $A_i(x)$ is increasing, the limit exists and
$\mathcal{M}_{\tilde{x_0}}(-\alpha)\in T:=
\mathcal{M}_{\tilde{x_0}}(\mathbb{R}^m)\ $. The dimension of $T$
is at most $m-1$ and we prove (\ref{6}) by claiming that $T$ dose
not intersect $\mathcal{P}^o$. Because if (\ref{6}) is not
satisfied, by our claim, the only possibility is that $T$ is a
subspace intersecting $\partial \mathcal{P}$ and containing
$\lim_{x\rightarrow \tilde{x_0}}A(x)= (\lim_{x\rightarrow
\tilde{x_0}}A_1(x), ..., \lim_{x\rightarrow \tilde{x_0}}A_m(x))$.
But then for some $i$, $A_i(\tilde{x_0})=0$ which contradicts the
fact that $A_i(\tilde{x_0})> A(x)>0$ for $\tilde{x_0}>x>0$.

Now we prove the claim. Let $\mathcal{S}^-(x)= \cup_{j=1}^m \{
\sum_{i=1, i\neq j}^m t_iM(x)(-e_i): t_i\geq 0 \}$. We have
$\mathcal{S}^-(0)= \mathcal{P}$ and by similar argument as we did
for $\mathcal{S}^+(x)$, $\mathcal{P}$ will be on one side of
$\mathbb{R}^m\backslash \mathcal{S}^-(x)$ for small $x$. Our claim
is established if we show that $\mathcal{S}^+(x)\cap
\mathcal{P}^o=\emptyset$ for $0<x<\tilde{x_0}$. If it was not the
case, then by continuity there is $x_1$, $0<x_1<\tilde{x_0}$ such
that $\mathcal{S}^+(x_1)$ intersects $\mathcal{S}^-(x_1)$ in at
least one nonzero vector. That means there is $v\neq 0$ such that
$v=\sum_{t_i\geq 0}t_iM(x_1)e_i=\sum_{t_i\geq 0}t'_iM(x_1)(-e_i)$
or equivalently $\sum_{t_i,t'_i\geq 0}(t_i+t'_i)M(x_1)e_i=0$ which
implies that $M(x)$ is nonsingular for $x_1<\tilde{x_0}$ which is
absurd.
\end{proof}
\begin{cor}\label{cor1}
Suppose $(Y(G),T)$ is an RFT and $f \in \mathcal{F}^o(Y(G))$.
Then $\phi_{G,f,w}(x)$ is $C^1$.
\end{cor}
\begin{proof}
This is clear if $r(\phi_{G,f,w})=0$. Otherwise the proof follows
from the fact that $F_{f,V}(x)$ is $C^1$ (see proof of Theorem 2
in \cite{P}). Because then $\alpha_{ij}(x)$ is $C^1$ and since
$M(x)$ is invertible, $A_i(x)$ is a rational map on
$\alpha_{ij}(x)$'s and hence $C^1$. Now since $\phi_{G,f,w}(x)$ is
a polynomial on $\alpha_{ij}(x)$'s and $A_i(x)$'s, must be $C^1$.
%Otherwise, by the above Theorem if $0\leq x <r(\phi_{G,f,w})$
%then $M(x)$ is invertible. Moreover, $F_{f,V}(x)$ is $C^1$ (see
%proof of Theorem 2 in \cite{P}). This implies that
%$\alpha_{ij}(x)$ and $A_i(x)$, as a rational map on
%$\alpha_{ij}(x)$'s are $C^1$. Now $\phi_{G,f,w}(x)$ is $C^1$
%which follows from its definition in (\ref{A}).
\end{proof}
Still some other results may be interesting. For instance for
$1\leq i\leq m$ write (\ref{A}) as
\begin{eqnarray*}
-\alpha_{i0}(x)=\alpha_{i1}(x)A_{1}(x) + ... +
(\alpha_{ii}(x)-1)A_{i}(x)+ ... + \alpha_{im}(x)A_{m}(x),
\end{eqnarray*}
 and note that $\alpha_{ij}(x)$ and $A_i(x)$
are non-negative. This implies  at least one coefficient of
$A_i(x)$ on right, in our case $(\alpha_{ii}(x)-1)$, must be
negative or $\alpha_{ii}(x)<1$. So if $\alpha_{ii}(x)>0$ for all
$i$, then $F_{f,V}=\sum_{i=1}^{m}\alpha_i(x)$ is uniformly
bounded on the domain of $\phi_{G,f,w}(x)$. Though this last
result holds anytime
if $\tilde{x_0}< r(F_{f,V})$.% In fact
%\begin{cor}\label{cor1}
%If $0\leq x<r(\phi_{G,f,w})$, then $\alpha_{ii}(x)< 1$ and
%$F_{f,V}$ is uniformly bounded.\hfill $\Box$
%\end{cor}
\begin{rem}\label{rem1} Let ($Y(G), T$) be a topological Markov chain
and $f\in \mathcal{F}^o(Y(G))$ with generating function
$\phi_{G,f,w}(x)$. Also suppose $h(T_f)$ the topological entropy
of the flow $T_f$ on $Y(G)$ is not infinity. Then by a result in
\cite{S}, for an arbitrary $w \in V(G)$
$$h(T_f)=\inf \left\{ h\geq 0 : \sum_{\gamma \in C(G;w)}e^{-hf^*(\gamma)}\leq 1 \right\}.$$
 Since $\sum_{\gamma \in
C(G;w)}e^{-hf^*(\gamma)}=\phi_{G,f,w}(e^{-h})$, we have
$h(T_f)=-\ln(\hat{x}_f)$ where $\hat{x}_f =\sup \{ x\geq 0 :
\phi_{G,f,w}(x)\leq 1 \}.$ Therefore, the problem of computing
$h(T_f)$ reduces to find $\hat{x}_f$. By the fact that
$\phi_{G,f,w}(x)$ is increasing, then $\hat{x}_f$ is either the
unique solution of $\phi_{G,f,w}(x)=1$ or
$\hat{x}_f=r(\phi_{G,f,w}).$
\end{rem}
\section{Applications}\label{egs}
In this section we give three examples. The first two arises  in
the study of the geodesic flows in the modular surface. The first
is a local perturbation of a TBS and our goal is to compare the
algorithm given in \cite{P} and the one deduced from Theorem
\ref{det}. For this reason, we choose exactly the same example
appearing as Example 1 in \cite{P}. We only produce
$\phi_{G,f,w}(x)$ and refer the reader to \cite{P} for the facts
behind this example and also seeing how this function is applied
to obtain an approximation for entropy of the respective
dynamical system.

The second example with some more details is not a local
perturbation of a countable TBS.
 For that example $\phi_{G,f,w}(x)$ is obtained and an estimate of
the entropy will be given. The last example illustrates a case where $r(\phi_{G,f,w})=r(F_{f,V})$.\\
\\
{\bf Example 1.} Let $V(G)=\{ 3,4,5,6,... \}$, $w=3$ and take the
set of
 forbidden edges to be
\begin{equation*}
  D=\{(3,3) , (3,4) , (3,5) , (4,3) , (5,3) \}.
\end{equation*}
This example occurs in the coding of geodesic flows on the modular
surface.  See Example 1 in \cite{P} for a brief explanation.
 There the following
formula is defined
\begin{equation}\label{p345}
\phi_{G,f,3}(x) =
\frac{x^{f(3)}(F_{f,V}(x)-x^{f(3)}-x^{f(4)}-x^{f(5)})(1-x^{f(4)}-x^{f(5)})}{1+x^{f(3)}-F_{f,V}(x)}
\end{equation}
for the case when denominator is positive.

By the above method the relation $\rho$ on $V(G)$, is
\begin{equation*}
  W_{\{ 3 \}}= \{ V_0=\{3\} , V_1 = \{ 4,5 \}, V_2= V(G) \backslash \{ 3,4,5 \}
  \}.
\end{equation*}
Hence $m=2$ and $\phi_{G , f,3}(x)= x^{f(3)}A_2(x)$. Statement
(\ref{A}) in the above theorem implies
$$\left\{
\begin{tabular}{l}
$A_1(x)= \alpha_{10}(x) +\alpha_{11}A_1 (x)+ \alpha_{12}A_2 (x)= (x^{f(4)}+x^{f(5)})(A_1(x)+A_2(x)), $ \\
$A_2(x)= \alpha_{20}(x) +\alpha_{21}A_1(x) + \alpha_{22}A_2(x) =
(\sum_{v \in V_2}x^{f(v)})(A_1(x)+A_2(x)+1),$
\end{tabular}
\right.$$ where all $\alpha_{ij}=\alpha_{ij}(x)$ are real functions
and in fact $\alpha_{11}=\alpha_{12}=x^{f(4)}+x^{f(5)}$,
$\alpha_{21}=\alpha_{22}=\sum_{i=6}^{\infty}x^{f(i)}$ and
$\alpha_{20}=x^{f(3)}$. So, $M=\begin{pmatrix}
  \alpha_{11}-1 & \alpha_{12} \\
  \alpha_{21} & \alpha_{22}-1
\end{pmatrix}$ and if $\det M\neq 0$, then $\begin{pmatrix}
  A_{1}(x) \\
  A_{2}(x)
\end{pmatrix}= M^{-1}\begin{pmatrix}
 0 \\
  -\alpha_{20}
\end{pmatrix}$.
Therefore, $A_1(x)=
\frac{\alpha_{12}\alpha_{20}}{1-\alpha_{11}-\alpha_{22}}$,
$A_2(x)=\frac{\alpha_{20}(1-\alpha_{11})}{1-\alpha_{11}-\alpha_{22}}.$
 But $1-\alpha_{11}-\alpha_{21}>0$, because $A_i$'s and $\alpha_{ij}$'s
 are positive by definition and $1-\alpha_{11}>0$ by the results following Corollary  (\ref{cor1}). That means
$\phi_{G,f,3}(x)=\frac{x^{f(3)}\alpha_{20}(x)(1-\alpha_{11}(x))}{1-\alpha_{11}(x)-\alpha_{22}(x)}$.
By evaluating $\alpha_{ij}(x)$, the formula (\ref{p345}) will be
established. Note that $\det M= 1-\alpha_{11}(x)-\alpha_{22}(x)=
1+x^{f(3)}-F_{f,V}(x)$ which is the denominator in (\ref{p345}).
\\
\\
\noindent {\bf Example 2.} Recall from \cite{K} that any
bi-infinite sequence of non-zero integers $\{ ...,v_{-1}, v_0,
v_1, ... \}$, $|v_i|\neq 1$ such that
$|\frac{1}{v}+\frac{1}{v'}|\leq \frac{1}{2}$ is realized as a
geometric code of an oriented  geodesic on the modular surface.
These codes are produced by choosing a suitable
 cross section, that is, a set which is hit infinitely many times in past and future by geodesic. So let $V(G)=\{v\in \mathbb{Z}: |v| \geq 2 \}$
and
\begin{eqnarray*}
    D_1 &=&\{(-3,-3) , (-3,-4) , (-3,-5) , (-4,-3) , (-5,-3),  \\
    && \,\, \, (3,3), (3,4) , (3,5) , (4,3), (5,3)\},\\
    D_2 &=&\{ (-v,-2), (-2,-v), (v,2), (2,v) : v\geq 2 \}.
\end{eqnarray*}
Set $D=D_1 \cup D_2$, that is, $(v,v')\in E(G)$ if and only if
$|\frac{1}{v}+\frac{1}{v'}|\leq \frac{1}{2}.$ Let $X=\{\{
...,v_{-1}, v_0, v_1, ... \}:|v_i|\geq 2,
|\frac{1}{v}+\frac{1}{v'}|\leq \frac{1}{2}\} $,
$\sigma(v_i)=v_{i+1}$ and $(X,\sigma)$ the associated system. In
fact, $(X\cup \{ ..., 1, -1, 1, -1, ... \},\sigma)$ is the
maximal 1-step countable topological Markov chain in the set of
all admissible codes known as geometric codes [3, Theorem 2.3].

By putting $w=2$, $W_{\{2\}}=\{V_0=\{2\}$, $V_1=\{3\}$, $V_2=\{
4,5\}$, $V_3=\{6,7,...\}$, $V_4=\{-2\}$, $V_5=\{-3 \}$, $V_6=\{
-4,-5\}$, $V_7=\{..., -7, -6\}\}$, one sees that $(X,\sigma)$ is
an RFT. Hence we apply our technique to give an estimation for
the entropy of $(X,\sigma)$.

Define $f(\{ ...,v_{-1}, v_0, v_1, ... \})=2\ln |cv_0|$, $c=1.25$
and let $\sigma_f$ be the special flow over $X$ with the ceiling
function $f$. (To keep the continuity of argument, we later give
some explanation  to justify choosing such an $f$.) We may assume
$f$ is defined on $V(G)$ and $f(v)=2\ln |cv|$.

 Note that
$\phi_{G,f,2}(x)= x^{f(2)}(A_4(x)+A_5(x)+A_6(x)+A_7(x))$. Also if
$\alpha_{i}=\alpha_{i}(x)$ and $A_i=A_i(x)$, then $A_i$'s are the
solution of the follower set of equations.
$$\left\{
\begin{tabular}{l}
$A_1= \alpha_{1}(A_3 +A_4+A_5+A_6+A_7)$ \\
$A_2= \alpha_{2}(A_2+ A_3 +A_4+A_5+A_6+A_7)$ \\
$A_3= \alpha_{3}(A_1+A_2+A_3 +A_4+A_5+A_6+A_7)$ \\
$A_4= \alpha_{4}(1 +A_1+A_2+A_3)$\\
$A_5= \alpha_{5}(1 +A_1+A_2+A_3+A_7)$ \\
$A_6= \alpha_{6}(1 +A_1+A_2+A_3 +A_6+A_7)$ \\
$A_7= \alpha_{7}(1 +A_1+A_2+A_3 +A_5+A_6+A_7).$ \\
\end{tabular}
\right.$$ Here $\alpha_0=\alpha_4=x^{f(2)}$,
$\alpha_1=\alpha_5=x^{f(3)}$,
$\alpha_2=\alpha_6=x^{f(4)}+x^{f(5)}$ and
$\alpha_3=\alpha_7=\sum_{v=6}^{\infty}x^{f(v)}$. Therefore, the
entropy will be $-\ln \hat{x}_f=0.8665$ where $\hat{x}_f$ is the
unique solution of  $\phi_{G,f,2}(x)=1$. (We used the computer
software Maple to perform the computations.)

Now we explain why such an $f$ was chosen above.  Let $x=\{...,
v_{-1}, v_0, v_1, ... \}$ be a geometric code for an oriented
geodesic $\gamma$. Then $w(x)=v_0(x)
-\frac{1}{v_1(x)-\frac{1}{\ldots}}$ called \textit{minus
continued fraction}, represents the attractive end point of
$\gamma$. Let $h$, the ceiling function, be the \textit{first
return time function} of oriented geodesic. This $h$ records the
time between two hits of cross sections by geodesic and is
cohomologous to $g(x)=2\ln |w(x)|$. Since  two cohomologous
ceiling functions give the same entropy for special flows on the
same base space, take $g$ to be the ceiling function over $X$.
Note that if $|v_i|=2$ and $|\frac{1}{v_i} +
\frac{1}{v_{i+1}}|\leq \frac{1}{2}$, then $v_i v_{i+1}<0$ and
$|v_{i+1}|$ can be arbitrary large. Therefore,
$$|w(x)|\leq |v_0(x)|+|\frac{1}{v_1(x)-\frac{1}{v_2(x)-\frac{1}{\ldots}}}|\leq |v_0(x)|+ \frac{1}{2}\leq |v_0(x)|+ \frac{|v_0(x)|}{4}.$$
This in turn shows that $|w(x)|\leq1.25|v_0(x)|=c|v_0(x)|.$ But,
this implies $g(x)\leq f(x)$ and hence $h(T_g)\geq h(T_f)\cong
0.8665$.

We like  to mention that our estimate improves slightly the
estimate obtained in \cite{K}.  There they consider $X'=\{
(...,v_{-1},v_0,v_1,...) :|v_i|\geq 3, v_i\in \Bbb{Z}\subset
X\}$. Then $\sigma$ is invariant on $X'$ and let
$\sigma'=\sigma_{|_X}$. Let $V'(G)=\{ v\in V:  |v_i|\geq 3
\}\subseteq V(G)$ and $D'=D_1$. It is proved in \cite{K} that the
special flow associated to $(X',\sigma')$ is a local perturbation
of a countable TBS and based on the results in \cite{P}, they
estimate the entropy  to be greater than 0.84171.

Recall that entropy of the geodesic flows in modular surface is 1
\cite{GK}, if we roughly  agree that bigger entropies of
subsystems are due to richer dynamics, hence $(X\cup \{ ..., 1,
-1, 1, -1, ... \},\sigma)$ with entropy greater than 0.8665 is a
fairly rich subsystem of
the geodesic flows in modular surface.\\
%Here we give an example so that $r(\phi_{G,f,w})=r(F_{f,V})$.\\
{\bf Example 3.} Let $G$ be a graph with vertex set
$V(G)=\{v_0=1, v_1=2,  ... \}$ and edge set $E(G)=\{ (v_i,v_j):\
v_i\neq v_j \ \textrm{and either}\ v_i \ \textrm{or}\ v_j\
\textrm{is}\ v_0 \}$. Let $w=\{ 1 \}$. Then $W_{\{ 1 \}}=\{
V_0=\{1\}, V_1=V(G)-\{1\} \}$. So $M(x)=[-1]$ which is invertible
for $0\leq x \leq r(F_{f,V})$. Therefore, by Theorem \ref{det},
$r(\phi_{G,f,1})=r(F_{f,V})$. Also
$\phi_{G,f,1}(x)=x^{f(1)}(\sum_{v\in V(G)-\{1\}}x^{f(v)})$.

Now let $f$ be an increasing function such that takes the value
$1$ on $v_0=1$ and value $k\in \mathbb{N}$, $k\geq 2$ exactly
$\lfloor\frac{2^{k}}{k^{2}}\rfloor$ times, where $\lfloor
r\rfloor$ denotes the integer part of $r$. Then
$r(F_{f,V})=\frac{1}{2}$ and $\phi_{G,f,1}(x)=x^{1}(\sum_{v\in
V(G)-\{1\}}x^{f(v)})=
x(\sum_{k=2}^{\infty}\lfloor\frac{2^{k}}{k^{2}}\rfloor x^k)\leq
\phi_{G,f,1}(\frac{1}{2})<0.85$. Hence $\hat{x}_f =\frac{1}{2}$.
See Remark (\ref{rem1}).
%savchenko
\section{An equivalent formula for generating function}\label{4}
%One of the basic tasks in \cite{P} and here is to evaluate
%$\phi_{G,f,w}(x)$ and in both cases one looks for some recursive
%formulas via an arrangement along the paths where a vertex
%returns to itself. This is mainly what we did in section
%\ref{generating} and obtained \cite{P}.
%
%Let $r(M)=\sup\{ x>0: \det M(x)\neq 0 \}$ where $M$ is the matrix
%introduced in Remark \ref{M}. We expect that $r(\phi_{G,f,w})=r(M)$,
%however, we do not rule out that $r(\phi_{G,f,w})$ may be greater
%that $r(M).$ We take this into consideration and give  a new
%approach. However, we are not sure that this new approach gives a
%larger radius of convergence for $\phi_{G,f,w}.$ But  as a
%consequence, we show explicitly that our method for the generating
%function implies exactly the one given in \cite{P}.
In this section by a rather new approach we give an explicit
formula for $\phi_{G,f,w}(x)$ where in the special case of local
perturbation of a TBS, the formula will exactly be the one given
in \cite{P}. (See corollary \ref{thepoly}). Let $W_{\{w\}} = \{
V_0=\{w\},V_1, ..., V_m\}$ be the partition of $V(G)$ as before.
Let $\widetilde{V}_G= \{ v \in V: \exists v' \in V \ni (v,v')
\not \in E(G) \}\setminus \{w\}$. We reindex $V_i$'s so that for
some $1\leq \ell \leq m$,
\begin{eqnarray}\label{V}
\left\{
\begin{tabular}{lllll}
$\widetilde{V}_G$ &=& $V_1 \cup ... \cup V_{\ell}$;&& \\
 $\cup_{i=\ell +1}^{m-1}V_i$ &=& $\{ v \in V(G)$&:&$\!\!\!\!\!\!(v,v')\in E(G),\
v'\in V(G),$\\
&& &&$\!\!\!\!\!\!(v'',v)\not \in E(G), \exists v'' \in V(G)\}$;\\
 $ V_m$&=& $\{v\in V(G)$ &:& $\!\!\!\!\!\!(v,v')\in
E(G),(v',v)\in E(G),$\\
&& && $\!\!\!\!\!\! v'\in V(G)\}.$
\end{tabular}
\right.
\end{eqnarray}
%Denote the  weighted adjacent matrix of $G$ with $A_G$. Then the
%weighted adjacent  matrix of the graph $H$, $A_H$, is called the
%\textit{reduced matrix} of $A_{G}$. Let $B^H$ be the $\ell \times
%\ell$ sub-matrix of $A_H$ on the upper left corner or $A_H$.
Let $M=M(x)$ be the matrix in (\ref{M}) which is obtained from
$W_{\{ w\}}\setminus \{w\}$. Let $\ell$ be as in (\ref{V}) and
\begin{eqnarray}\label{C}
C=\begin{pmatrix}
  \alpha_{11}-1 & \alpha_{12} & \ldots & \alpha_{1\ell} \\
  \alpha_{21} & \alpha_{22}-1 & \ldots & \alpha_{2\ell} \\
  \vdots &  &  &  \\
  \alpha_{\ell 1} & \alpha_{\ell 2} & \ldots & \alpha_{\ell \ell}-1
\end{pmatrix}\end{eqnarray}
an $\ell \times \ell$ sub-matrix of $M$ on the upper left corner.
Note that $B=C+Id$ represents a weighted adjacent matrix for the
vertices of $\widetilde{V}_G$. The next lemma shows when
$\phi_{G,f,w}(x)$ is defined, then $C$ is invertible.

\begin{lem}
Suppose $x_C$ is the smallest real number such that $\det C=0$.
Then $x_C \geq r(\phi_{G,f,w})$.
\end{lem}
\begin{proof}
We have
\begin{eqnarray}\label{B}
\sum_{n=0}^{\infty}[B(x)]_{V_iV_j}^{(n)}=
\frac{(-1)^{e(V_i)+e(V_j)}\det ((\textrm{Id}\ -
B(x))_{e(V_j)e(V_i)})}{\det(\textrm{Id}-B(x))}
\end{eqnarray}
where $(\textrm{Id}\ - B(x))_{e(V_j)e(V_i)}$ is the sub-matrix of
$\textrm{Id} - B(x)=-C$ obtained by deleting $V_j$th row and
$V_i$th column. In fact (\ref{B}) is the same identity as (2.18)
in \cite{P}. It holds because if we let $A=[a_{ij}]$ be a $p
\times p$ matrix over $\mathbb{C}$ and if $a_{ij}^{n}$ be the
$ij$th entry of matrix $A^{(n)}$ then $\sum_{n\geq
0}a_{ij}^{(n)}z^n= \frac{(-1)^{i+j}\det
((\textrm{Id}-zA)_{ji})}{\det(\textrm{Id}-zA)}$, $1\leq i,j\leq
p$.

Note that the right hand of (\ref{B}) is by definition
$[C^{-1}]_{V_iV_j}$, the $V_iV_j$th entry of $C^{-1}$. This also
shows that the radius of convergence of the series on the left is
$x_C$. But $\sum_{n=0}^{\infty}[B(x)]_{V_iV_j}^{(n)}= \sum
x^{f^*(\gamma)}$ where the sum on the right is over the paths
$\gamma$ from all $v_i$ in $V_i$ to all $v_j$ in $V_j$ with all
possible lengths. If there is not any path from $V_i$ to $V_j$ in
$\widetilde{V}_G$ then
$\sum_{n=0}^{\infty}[B(x)]_{V_iV_j}^{(n)}=0$. Let $H$ be the
reduced graph of $G$. Since the graph $H$ with vertices $W_{\{
w\}}$ is connected so we can fix one path $\gamma_1$ from $w$ to
$v_i\in V_i$ and one path $\gamma_2$ from $v_j\in V_j$ to $w$.
Then $\gamma_1 \gamma \gamma_2 \in C(G;w)$ and $\sum_{V_i,V_j \in
\widetilde{V}_G}\sum_{\gamma \in [B(x)]_{V_iV_j}^n}
x^{f^*(\gamma_1 \gamma \gamma_2)} \leq \phi_{G,f,w}(x)$which
implies $x_C \geq r(\phi_{G,f,w})$.
%But if
%for any $x$, $\det M(x)=0$, then $r(\phi_{G,f,w})=x_M$ by Theorem
%\ref{det}.
\end{proof}

Recall that $\alpha_i= \alpha_i(x)=\sum_{v\in V_i}x^{f(v)}$,
 $1\leq i \leq m$. Then  $A_i$'s satisfy (\ref{A}) or equivalently
\begin{eqnarray}\label{theta}
\left\{
\begin{tabular}{l}
$A_1= \alpha_{10} +\alpha_{11}A_1 + ... +\alpha_{1m}A_m$ \\
$A_2= \alpha_{20} +\alpha_{21}A_1 + ... +\alpha_{2m}A_m $ \\
$\vdots$\\
$A_{\ell}= \alpha_{\ell 0} +\alpha_{\ell 1}A_1 + ... +\alpha_{\ell m}A_m $ \\
$A_{\ell +1}=\alpha_{\ell +1}(1+A_{1}+A_{2}+...+A_{m})$\\
$\vdots$\\
$A_m=\alpha_{m}(1+A_{1}+A_{2}+...+A_{m}). $
\end{tabular}
\right.
\end{eqnarray}
 Note that $\alpha_i(x) \not \equiv 0$ for $i\geq m_0$. Let $\zeta=\zeta(\ell,m):= \alpha_{\ell +1}+...+\alpha_{m-1}$
and $F_{f,V,V_i}=F_{f,V,V_i}(x)=\sum_{v\in
V_i^+-\widetilde{V}_G}x^{f(v)}.$ %Let
%\begin{eqnarray}\label{C}
%C=\begin{pmatrix}
%  \alpha_{11}-1 & \alpha_{12} & \ldots & \alpha_{1\ell} \\
%  \alpha_{21} & \alpha_{22}-1 & \ldots & \alpha_{2\ell} \\
%  \vdots &  &  &  \\
%  \alpha_{\ell 1} & \alpha_{\ell 2} & \ldots & \alpha_{\ell \ell}-1
%\end{pmatrix}\end{eqnarray} be the upper left $\ell \times \ell$ sub-matrix of $M$. Here we will not necessarily assume that $M$ is
%non-singular. However, we need $C$ be non-singular in the domain
%of $\phi_{G,f,w}.$
  Denote by $\langle.\, ,.\rangle$ the standard dot product of two
vectors and $Row_i(N)$  the $i$th row of matrix $N$. Set
\begin{eqnarray*}
\alpha_{H,f,w}(x)  &=& \sum_{V_i\subseteq \widetilde{V}_G  \atop{(w,V_i)\in E(H)}}\langle Row_{i}(C^{-1}),[-\alpha_1F_{f,V,V_1},...,-\alpha_{\ell} F_{f,V,V_{\ell}}]\rangle,\\
\alpha_{H,f,\widetilde{V}_G}(x)  &=& \sum_{V_i\subseteq \widetilde{V}_G}\langle Row_{i}(C^{-1}),[-\alpha_1F_{f,V,V_1},...,-\alpha_{\ell} F_{f,V,V_{\ell}}]\rangle,\\
\sigma_{H,f,w}(x) &=& \sum_{V_i \subseteq \widetilde{V}_G}\langle Row_{i}(C^{-1}), [-\alpha_{10},...,-\alpha_{\ell 0}]\rangle,\\
\widetilde{\phi}_{H,f,w}(x)&=&\sum_{V_i\subseteq \widetilde{V}_G
\atop{(w,V_i)\in E(H)}}\langle
Row_{i}(C^{-1}),[-\alpha_{10},...,-\alpha_{\ell
0}]\rangle+\alpha_{00}.
\end{eqnarray*}
\begin{The}\label{equv}
Let ($Y(G),T$) be an RFT and $f \in \mathcal{F}^0(Y(G_0))$. Then
for  $w\in V(G)$
\begin{eqnarray*}\phi_{G,f,w}(x)&=&\!\!\!\!x^{f(w)}(F_{f,V,w}(x)+\alpha_{H,f,w}(x))\left(\frac{\sigma_{H,f,w}(x)+1}{1-\zeta(x)-\alpha_m(x)-\alpha_{H,f,\widetilde{V}_G}(x)}\right)\\
&&+\widetilde{\phi}_{H,f,w}(x)
\end{eqnarray*} for those $0\leq x <r(F_{f,V})$
where
\begin{equation}\label{den}
1-\zeta(x)-\alpha_m(x)-\alpha_{H,f,\widetilde{V}_G}(x)>0.
\end{equation}

\end{The}
\begin{proof}
Let $S_1=S_1(x):=A_1(x)+ ... +A_{\ell}(x)$ and
$S_2=S_2(x):=A_{\ell +1}(x)+... +A_{m-1}(x)$. Then
$S_2=\zeta(1+S_1+S_2+A_m)$ or equivalently
$S_2=\frac{\zeta}{1-\zeta}(1+S_1+A_m)$. Applying $S_1$ and $S_2$
in the $m_0$th equation in (\ref{theta}),
$A_m=\frac{\alpha_m}{1-\zeta-\alpha_m}(S_1+1)$ which implies
$S_2=\frac{\zeta}{1-\zeta -\alpha_m}(S_1+1)$. By Evaluating $A_j$,
$j>\ell$ in terms of $S_1$ and $S_2$,
 and then evaluating $S_2$ and $A_m$ in terms of $S_1,
\zeta$, we will have
\begin{eqnarray*}C\begin{pmatrix}
  A_{1} \\
  \vdots \\
  A_{\ell}
\end{pmatrix}&=& \begin{pmatrix}
  -\alpha_{1(\ell+1)}A_{\ell+1}- ... -\alpha_{1(m-1)}A_{m-1}- \alpha_1 A_{m} - \alpha_{10} \\
  \vdots\\
 -\alpha_{\ell(\ell+1)}A_{\ell+1}- ... -\alpha_{\ell(m-1)}A_{m-1}- \alpha_{\ell} A_{m} - \alpha_{\ell 0}
\end{pmatrix}\\ &=&\begin{pmatrix}
  \xi_{1} \\
  \vdots\\
  \xi_{\ell}
\end{pmatrix}-\begin{pmatrix}
  \alpha_{10} \\
  \vdots \\
  \alpha_{\ell 0}
\end{pmatrix}.\end{eqnarray*}
where
\begin{eqnarray*}
   \xi_i &=& \sum_{k=\ell+1}^{m-1}-\alpha_{ik}A_k-\alpha_iA_m\\
   &=& \sum_{k=\ell+1}^{m-1}-\alpha_{ik}(\alpha_{k}+\alpha_{k}S_1
   +\alpha_{k}S_2
+\alpha_{k}A_{m})-\alpha_iA_m \\
     &=& -\sum_{k=\ell+1}^{m-1}(\alpha_{ik})\alpha_k(S_1+S_2+A_m+1)-\alpha_iA_m\\
     &=& -\frac{(S_1+1)(\sum_{k=\ell+1}^{m}\alpha_i(\alpha_{ik}))}{1-\zeta-\alpha_m}\\
     &=& -\frac{(S_1+1)(\alpha_i F_{f,V,V_i})}{1-\zeta-\alpha_m}.
\end{eqnarray*}
Therefore we have
\begin{eqnarray*}
\phi_{G,f,w}(x) &=& \alpha_{01}A_{1}+ ...+ \alpha_{0\ell}A_{\ell}+ \alpha_{0(\ell +1)}A_{\ell +1}+...+ \alpha_{0m}A_{m}+\alpha_{00}\\
&=& (\alpha_{01}A_{1}+ ...+
\alpha_{0\ell}A_{\ell})\\
&& + x^{f(w)}F_{f,V,w}(1+S_1+S_2+A_m)+\alpha_{00}.\end{eqnarray*}
This leads to
\begin{eqnarray*}\phi_{G,f,w}(x) &=&x^{f(w)}\sum_{V_i\subset
\widetilde{V}_G \atop{(w,V_i)\in E(H)}}\langle Row_{i}(C^{-1}),
[\xi_1,...,\xi_{\ell}]-[\alpha_{10},...,\alpha_{\ell 0}]\rangle\\
& & +x^{f(w)}F_{f,V,w}\left(1+S_1+\frac{\alpha_m+\zeta}{1-\zeta-\alpha_m}(1+S_1)\right)+\alpha_{00} \\
&=& x^{f(w)}\sum_{V_i \subset \widetilde{V}_G\atop{(w,V_i)\in
E(H)}}\langle Row_{i}(C^{-1}),
[\xi_1,...,\xi_{\ell}]\rangle\\
& &+x^{f(w)}\sum_{V_i\subset \widetilde{V}_G \atop{(w,V_i)\in
E(H)}}\langle Row_{i}(C^{-1}),
-[\alpha_{10},...,\alpha_{\ell 0}]\rangle\\
&&+x^{f(w)} \frac{(S_1+1)F_{f,V,w}}{1-\zeta-\alpha_m} +\alpha_{00}\\
&=&x^{f(w)}\frac{(S_1+1)}{1-\zeta-\alpha_m}\times \\
&&\left(\sum_{V_i\subseteq \widetilde{V}  \atop{(w,V_i)\in
E(H)}}\langle
Row_{i}(C^{-1}),[-\alpha_1F_{f,V,V_1},...,-\alpha_{\ell}
F_{f,V,V_{\ell}}]\rangle\right)\\ & &+x^{f(w)}
\frac{(S_1+1)F_{f,V,w}}{1-\zeta-\alpha_m}+\widetilde{\phi}_{H,f,w}(x)\\
&=&x^{f(w)}\frac{(S_1+1)}{1-\zeta-\alpha_m}\left(
\alpha_{H,,f,w}(x) +F_{f,V,w}\right)+\widetilde{\phi}_{H,f,w}(x).
\end{eqnarray*}
But
\begin{eqnarray*}
S_1&=& A_1 +... +A_{\ell}\\
&=& \sum_{i=1}^{\ell}\langle Row_{i}(C^{-1}),([\xi_1,...,\xi_{\ell}]-[\alpha_{10},...,\alpha_{\ell 0}])\rangle \\
&=&\frac{S_1+1}{1-\zeta-\alpha_{m}}\sum_{i=1}^{\ell}\langle Row_{i}(C^{-1}),[-\alpha_1F_{f,V,V_1},...,-\alpha_{\ell}F_{f,V,V_{\ell}}]\rangle\\
&&+\sigma_{H,f,w}(x)\\
&=&\frac{S_1+1}{1-\zeta-\alpha_{m}}\alpha_{H,f,\widetilde{V}_G}(x)+\sigma_{H,f,w}(x).
\end{eqnarray*}
That means
$$S_1=\frac{\alpha_{H,f,\widetilde{V}_G}(x)+\sigma_{G,f,w}(x)(1-\zeta-\alpha_{m})}{1-\zeta-\alpha_{m}-\alpha_{H,f,\widetilde{V}_G}(x)}$$
where the denominator is positive. So
\begin{eqnarray*}
\phi_{G,f,w}(x)
&=& x^{f(w)}\left(F_{f,V,w}+\alpha_{H,f,w}(x)\right)\left(\frac{S_1+1}{1-\zeta-\alpha_m}\right)+\widetilde{\phi}_{H,f,w}(x)\\
&=&x^{f(w)}(F_{f,V,w}+\alpha_{H,f,w}(x))\left(\frac{\sigma_{G,f,w}(x)+1}{1-\zeta-\alpha_m-\alpha_{H,f,\widetilde{V}_G}(x)}\right)\\
&&+\widetilde{\phi}_{H_w,f,w}(x).
\end{eqnarray*}
Note that all the arguments are reversible and the proof is
established.
\end{proof}
The following Theorem relates the condition (\ref{den}) to the
conclusion of Theorem \ref{det}.
\begin{The}
Let $M$ and $C$ be as before. Then
$$\frac{\det M}{(-1)^{m-\ell}\det C} = 1-\zeta-\alpha_m-\alpha_{H,f,\widetilde{V}_G}.$$
\end{The}
Note that the statement on right is the same statement appearing
in (\ref{phi}).
\begin{proof}
For $k\leq m$ let $X_{(k)}=(x_1, ..., x_k )$ and
$b_k=(\alpha_{01}, ..., \alpha_{0k})$. Recall $N_{e(i)e(j)}$ is
the sub-matrix of $N$ obtained by deleting its $i$th row and $j$th
column. Also let $N_{e(j)}$ be the matrix obtained from an $n
\times n$ matrix $N$ by replacing its $j$th column with $b_n$.
Consider $NX_{(k)}=-b_{k}$. Then by using Cramer's rule, we have
$x_{i}= \frac{\det N_{e(i)}}{\det N}$ or $\det N_{e(i)}= x_i \det
N$, $1\leq i\leq n$.

The proof is based on the induction on the last $m-\ell$ rows of
$M$. First let $M=C$. Then $\zeta, \alpha_m$ and
$\alpha_{H,f,\widetilde{V}_G}(x)$ are all zero and the conclusion
holds trivially. If we take $m=\ell +1$. Then by choosing the last
row of $M$ for computing the determinant of $M$ and then using the
Cramer's rule we have
\begin{eqnarray*}
\frac{\det M}{-\det C} &=&\frac{1}{-\det C}\left((-1)^{m+1}\alpha_m\det M_{e(m)e(1)}+\ldots\right. \\
&& + (-1)^{m+\ell}\alpha_m\det M_{e(m)e(\ell)}+(-1)^{2m}(\alpha_m-1)\det M_{e(m)e(m)}\left.\right)\\
&=& \frac{\alpha_m}{-\det C}\left((-1)^{m+1}(-1)^{m-1}\det C_{e(1)}+ \ldots \right.\\
&&+ (-1)^{m+\ell}(-1)^{m-\ell}\det C_{e(\ell)} \left. \right)+(1-\alpha_m)\\
&=& \frac{\alpha_m}{-\det C} (x_1\det C +\ldots + x_{\ell}\det C)+ (1-\alpha_m)\\
&=& -\alpha_m (\langle \mathrm{Row}_1 C^{-1},-b_{\ell}\rangle+\ldots + \langle \mathrm{Row}_{\ell} C^{-1},-b_{\ell}\rangle)\\
&& +(1-\alpha_m)\\
&=& -\sum_{i=1}^{\ell}\langle \mathrm{Row}_i C^{-1}, [-\alpha_1 F_{f,V,V_1}, \ldots , -\alpha_{\ell}F_{f,V,V_{\ell}}]\rangle +(1-\alpha_m) \\
&=& -\alpha_{H,f,\widetilde{V}_G}(x) +(1-\alpha_m). \\
\end{eqnarray*}
To emphasize the dependent of $\zeta$, $F_{f,V,V_i}$ and
$\alpha_{H,f,\widetilde{V}_G}$ on $m$, we will show them by
$\zeta^{(m)}$, $F_{f,V,V_i}^{(m)}$ and
$\alpha_{H,f,\widetilde{V}_G}^{(m)}$. Now let $m_0=\ell +k$, $k>1$
and assume
\begin{eqnarray}\label{farz}\frac{\det M}{(-1)^{k}\det C} = 1-\zeta^{(m_0)}\alpha_{m_0}-\alpha^{(m_0)}_{H,f,\widetilde{V}_G}.
\end{eqnarray}
We prove the formula for $m=m_0+1$. Again we apply Cramer's rule,
then
\begin{eqnarray*}
\frac{\det M}{(-1)^{k+1}\det C}&=&\frac{(-1)^{2m}}{(-1)^{k+1}\det
C} \left( \alpha_m \det M_{e(m)e(1)}+\cdots  \right.\\
&&+\alpha_m \det M_{e(m)e(m_0)}+(\alpha_{m}-1)\det M_{e(m)e(m)}\left.\right)\\
&=& \frac{\alpha_m \det M_{e(m)e(m)}}{(-1)^{k+1}\det C}(x_1+\cdots
+ x_{m_0})\\
&&+ \frac{\det M_{e(m)e(m)}}{(-1)^{k+1}\det C}(\alpha_m -1).
\end{eqnarray*}
%\begin{eqnarray*}
%\frac{\det M}{(-1)^{k+1}\det C}&=&\frac{1}{(-1)^{k+1}\det C}\\
%&(&-1)^{2m}\alpha_m \det M_{e(m)e(1)}+\cdots +(-1)^{2m}\alpha_m
%\det M_{e(m)e(m_0)} +
%&(&-1)^{2m}(\alpha_{m}-1)\det M_{e(m)e(m)}\\
%&& &=& \frac{\alpha_m \det M_{e(m)e(m)}}{(-1)^{k+1}\det
%C}(x_1+\cdots + x_{m_0})+ \frac{\det M_{e(m)e(m)}}{(-1)^{k+1}\det
%C}(\alpha_m -1).
%\end{eqnarray*}
Compute $x_i$, $0\leq i \leq m_0$ from the set of equations
 $M_{e(m)e(m)}X_{(m_0)}=-b_{m_0}$.

It is convenient to let $S_1=x_1+...+x_{\ell}$ and
$S_2=x_{\ell+1}+...+x_{m_0}$. The first $\ell$ equations can be
written as
\begin{eqnarray*}C\begin{pmatrix}
  x_{1} \\
  \vdots \\
  x_{\ell}
\end{pmatrix}&=& \begin{pmatrix}
  -\alpha_{01}-\alpha_{1(\ell+1)}x_{\ell+1}- ... - \alpha_{1m_0} x_{m_0}\\
  \vdots\\
 - \alpha_{0\ell}-\alpha_{\ell(\ell+1)}x_{\ell+1}- ... - \alpha_{\ell m_0} x_{m_0}
\end{pmatrix}\\ &=&\begin{pmatrix}
  -\alpha_{01}-\alpha_1 (S_1+S_2+1)F^{(m_0)}_{f,V,V_1} \\
  \vdots\\
-\alpha_{0\ell}-\alpha_{\ell}(S_1+S_2+1)F^{(m_0)}_{f,V,V_{\ell}}
\end{pmatrix}.\end{eqnarray*}
So
\begin{eqnarray}\label{one}
S_1&=&\langle
\textrm{Row}_iC^{-1},[-\alpha_{01},...,-\alpha_{0\ell
}]\rangle\nonumber\\
&& +(S_1+S_2+1)\sum_{i=1}^{\ell}\langle
\textrm{Row}_iC^{-1},[-\alpha_1F^{(m_0)}_{V_1},...,-\alpha_{\ell}F^{(m_0)}_{V_{\ell}}]\rangle.
\end{eqnarray}
The rest of equations are
\begin{eqnarray}
\left\{
\begin{tabular}{l}
$\alpha_{\ell +1}x_1+...+(\alpha_{\ell +1}-1)x_{\ell +1}+ ... +\alpha_{\ell +1}x_{m_0}=-\alpha_{\ell +1}$ \\
$\vdots$\\
$\alpha_{m_0}x_1+...+(\alpha_{m_0 }-1)x_{\ell +1}+
...+\alpha_{m_0}x_{m_0}=-\alpha_{m_0}.$
\end{tabular}
\right.
\end{eqnarray}
So
\begin{equation}\label{two}
(\zeta^{(m_0)}) (x_1 + ... +x_{\ell}) + (\alpha_{\ell +1}+ ...
+\alpha_{m_0}-1)(x_{\ell +1}+ ... +x_{m_0})= -\zeta^{(m_0)}.
\end{equation}
From (\ref{one}) and (\ref{two}) we have
\begin{eqnarray*}S_1+S_2&=&
\frac{\langle \textrm{Row}_iC^{-1}, [-\alpha_{01}, ...,
-\alpha_{0\ell }]\rangle+ \zeta^{(m_0)}
}{D}\\
&&+ \frac{ \sum_{i=1}^{\ell}\langle
\textrm{Row}_iC^{-1},[-\alpha_1 F^{(m_0)}_{f,V,V_1} , ...
-\alpha_{\ell}
F^{(m_0)}_{f,V,V_{\ell}}]\rangle}{D},\end{eqnarray*} where
$D=1-\zeta^{(m_0)}-\sum_{i=1}^{\ell}\langle
\textrm{Row}_iC^{-1},[-\alpha_1 F^{(m_0)}_{f,V,V_1}, ... ,
-\alpha_{\ell}F^{(m_0)}_{f,V,V_{\ell}}]\rangle$. Replace
$M_{e(m)e(m)}$, for $M$ in (\ref{farz}). Then
\begin{eqnarray*}
\frac{\det M}{(-1)^{k+1}\det C}&=&\frac{\det
M_{e(m)e(m)}}{(-1)^{k+1}\det
C}(\alpha_m (S_1+S_2)+(\alpha_m -1))\\
&=& -\sum_{i=1}^{\ell}\langle \textrm{Row}_iC^{-1},[-\alpha_1 F^{(m)}_{f,V,V_1},..., -\alpha_{\ell}F^{(m)}_{f,V,V_{\ell}}]\rangle\\
&&+(1-\zeta^{(m)}-\alpha_m)\\
&=&1-\zeta^{(m)}-\alpha_m-\alpha^{(m)}_{H,f,\widetilde{V}_G}(x).
\end{eqnarray*}
\end{proof}
%\begin{rem}
%We may let $|W_{\{ w\}}|=\infty$. At least not much changes must
%be done for Theorem \ref{equv} and its proof except some cares for
%notations. For instance, define
%$\phi_{G,f,w}(x)=\alpha_{00}+\sum_{i=1}^{\infty}\alpha_{0i}A_i(x).$
%Also statements (\ref{V}) and (\ref{theta}) can be worked out
%accordingly. The matrix $C$ in (\ref{C}) will be infinite
%dimensional  and define $C^{-1}$ to be the matrix such that
%$C^{-1}C=\textrm{Id}.$ Then both theorem  and its proof need just
%some minor changes. However, we did concentrate on finite case.
%Because, when $|W_{\{ w\}}|=\infty$, the $\phi_{G,f,w}(x)$ will
%be so complicated that cannot be used in practice.
%\end{rem}

Now we are in a position to state that our result is identical to
those of [3, $\S$2] for a local perturbation of a TBS. Therefore,
we recall some notations from \cite{P}.  Suppose $(Y(G),T)$ is a
local perturbation of a countable TMC. Fix $w \in
\widetilde{V}_G$. By our earlier notations, $\widetilde{V}_G
\subseteq \widetilde{V}$ and $\widetilde{V}_G$ is the union of
some elements of $W_{\{w\}}=\{V_0=\{w\}, V_1 , ... , V_{m } \}$
which $V_i$, $1\leq i\leq \ell$ all must have finite elements
where $\ell$ is defined as (\ref{V}). For $v \in \widetilde{V}_G$
let $A_{G,v}=\{v' \in V: (v,v') \not \in E(G) \}$.  Let
$\widetilde{G}_w$ be the sub-graph of $G$ with $V(\widetilde{G}_w)
=\widetilde{V}_G \setminus \{w\}$ and $E(\widetilde{G}_w) =\{
(v,v')\in E(G): v,v' \in V(\widetilde{G}_w) \}$. Define a matrix
$B^{\widetilde{G}_w}(x)=(b^{\widetilde{G}_w}_{vv'}(x))$, $v,v'
\in V(\widetilde{G}_w)$ where $b^{\widetilde{G}_w}_{vv'}(x)$ is
$x^{f(v)}$ when $(v,v')\in E(\widetilde{G}_w)$ and zero
otherwise. Denote by $[B^{\widetilde{G}_w}]_{vv'}^{(n)}$ the
entry corresponding to the $v$th row and $v'$th column of $n$th
power of the matrix $B^{\widetilde{G}_w}$. Let
$\beta^{\widetilde{G}_w}_{vv'}(x) = \sum^{\infty}_{n=0}
[B^{\widetilde{G}_w}(x)]_{vv'}^{(n)}$, $v,v' \in
(\widetilde{G}_w)$ and  $F_{f,V,v}(x) = F_{f,V}(x) - \sum_{v' \in
(\widetilde{V}_{G} \cap \overline{A}_{G,v})\cup
A_{G,v}}x^{f(v')}$ where $v \in \widetilde{V}_{G}$ and the bar
over a set is the complement operation. Note that $F_{f,V,v_i}=
F_{f,V,V_j}$ for all $v_i \in V_j$. Finally for $U \subseteq
\widetilde{V}_G$ let
\begin{eqnarray*}
\alpha_{U,f,w}(x) &=& \sum_{v' \in U} \sum_{v'' \in
V(\widetilde{G}_w)} \beta^{\widetilde{G}_w}_{v'v''}(x) x^{f(v'')}
F_{f,V,v''}(x)\\
\sigma_{G,f,w}(x)&=&\sum_{v' \in V(\widetilde{G}_w)} \sum_{v'' \in
V(\widetilde{G}_w) \atop{(v'' , w) \in E(G)}}
\beta^{\widetilde{G}_w}_{v'v''}(x)x^{f(v'')}\\
\phi_{\widetilde{G},f,w}(x)&=&x^{f(w)}\sum_{v' \in
V(\widetilde{G}_w)\cap \overline{A}_{G,w}} \sum_{v'' \in
V(\widetilde{G}_w)\atop{ (v'',w)\in
E(G)}}\beta^{\widetilde{G}_w}_{v'v''}(x)x^{f(v'')} + I(w,w)
\end{eqnarray*}
where $I(w,w)=\alpha_{00}.$

Let $P=Id-B^{\widetilde{G}_w}$. Clearly, if $P$ is invertible then
$C$ is invertible.
\begin{lem}\label{redumat}
 Let $1\leq i\leq \ell$. Set $k=0$ if $i=1$ and $k=|V_1| +\cdots+ |V_{i-1}|$ if
 and $2\leq i\leq \ell$. Then
$$\sum_{j=k+1}^{k+|V_i|}
\langle
Row_j(P^{-1}),[b^{\widetilde{G}_{w}}_{10},\cdots,b^{\widetilde{G}_w}_{n0}]\rangle
 = \langle Row_i(C^{-1}),[-\alpha_{10},\cdots,-\alpha_{\ell 0}]\rangle.$$
\end{lem}
\begin{proof}
Suppose \begin{eqnarray}\label{Ct}C (s_1, ...,s_{\ell}) =
(-\alpha_{10}, \cdots , -\alpha_{\ell 0}),\end{eqnarray}
$P(r_{1}^1,\cdots,r_{|V_1|}^1,\cdots,r_{|V_{\ell}|}^{\ell})
=(b_{10}^{\widetilde{G}_w},\cdots,b_{n0}^{\widetilde{G}_w}).$
 It suffices to show
 $r^i_1+\cdots+r^i_{|V_{i}|}=s_i$, $1\leq i \leq \ell$.

For each $i$, $1\leq i \leq \ell$, $\alpha_{i1}s_1+ \ldots
+(\alpha_{ii}-1)s_i+\ldots + \alpha_{i\ell}s_{\ell}= -\alpha_{i0}$
and
$$\left\{\begin{tabular}{l}
 $ P_{(k+1)1}r_{1}^{1} +\cdots +  (P_{(k+1)(k+1)}-1)r_{1}^{i} +\cdots+  P_{(k+1)n}r_{|V_{\ell}|}^{\ell}= b_{(k+1)0}^{\widetilde{G}_w}$ \\
  $\vdots$\\
$P_{(k+|V_i|)1}r_{1}^{1}+\cdots+
(P_{(k+|V_i|)(k+|V_i|)}-1)r_{|V_i|}^{i}+\cdots+P_{(k+|V_i|)n}r_{|V_{\ell}|}^{\ell}$\\
$=b_{(k+|V_{i}|)0}^{\widetilde{G}_w}.$
\end{tabular}
\right.$$ By summing up all the above equations we will have
\begin{eqnarray*}
&&\!\! \!\! \!\!\!\! \!\! \!\!(P_{(k+1)1}+\cdots+P_{(k+|V_i|)1})r_1^1+\cdots +(P_{(k+1)(k+1)}+\cdots+P_{(k+|V_i|)(k+1)}-1)r_1^i\\
 &&+\cdots + (P_{(k+1)(k+|V_i|)}+\cdots+P_{(k+|V_i|)(k+|V_i|)}-1)r^i_{|V_i|}+\cdots\\
 &&+ (P_{(k+1)n}+...+P_{(k+|V_i|)n})r^{\ell}_{|V_{\ell}|}=b_{(k+1)0}^{\widetilde{G}_w}+\cdots+b_{(k+|V_{i}|)0}^{\widetilde{G}_w}\\
\end{eqnarray*}
and hence $(\alpha_{i1})r_1^1 +\cdots+
 (\alpha_{ii}-1)r_1^i+\cdots+(\alpha_{ii}-1)r^i_{|V_i|}+\cdots+(\alpha_{in})r_{|V_{\ell}|}^{\ell}=-\alpha_{i0}.$
 Therefore, $$\alpha_{i1}(r_1^1+\cdots+r^1_{|V_1|})+\cdots+ (\alpha_{ii}-1)(r_1^{i}+\cdots+
 r^{i}_{|V_i|})+\cdots+\alpha_{in}(r^{\ell}_{1}+\cdots+r^{\ell}_{|V_{\ell}|})=-\alpha_{i0}.$$
 Comparing this and (\ref{Ct}) one has $r^i_1+\cdots+r^i_{|V_{i}|}=s_i.$
\end{proof}
\begin{cor}\label{thepoly}
 Let $(Y(G),T)$ be a local perturbation of a
topological Bernoulli scheme $(Y(G_0), T)$ with a countable set
$V$ and let $f \in \mathcal{F}^0(Y(G))$ be a positive function.
Then for $w \in \widetilde{V}_G$
$$\phi_{G,f,w}(x) = \phi_{\widetilde{G},f,w}(x) + \frac{x^{f(w)}(F_{f,V,w}(x)+ \alpha_{V(\widetilde{G}_w)\cap \overline{A}_{G,w,f,w}}(x))(1+\sigma_{G,f,w}(x))}{1+\sum_{v' \in \widetilde{V}_G}x^{f(v')}-F_{f,V}(x)- \alpha_{V(\widetilde{G}_w),f,w}(x)}$$
for those $x \geq 0$ such that the denominator of last fraction
is positive [3, Theorem 2].\hfill $\Box$
\end{cor}
\begin{proof}
Let $w \in \widetilde{V}(G)$ and $W_{\{w\}}$ be the partition for
vertices of $V(G)$. Notice that here $|V_i|<\infty$, $1\leq i
\leq \ell$ and  $F_{f,V,v_i}= F_{f,V,V_j}$ for all $v_i \in V_j$.
Also from (\ref{B}), one has
\begin{eqnarray*}\beta_{v'v''}^{\widetilde{G}_w}(x)&=&
((-1)^{e(v')+e(v'')}
\det((I-B^{\widetilde{G}_w}(x))_{e(v'')e(v')}))/(\det(I-B^{\widetilde{G}_w}(x)))\\
&=&[P^{-1}]_{v'v''}.
\end{eqnarray*}
%From \cite{P}, $\beta_{v'v''}^{\widetilde{G}_w}(x)=
%((-1)^{e(v')+e(v'')}
%\det((I-B^{\widetilde{G}_w}(x))_{e(v'')e(v')}))/(\det(I-B^{\widetilde{G}_w}(x)))$
%where $((I-B^{\widetilde{G}_w}(x))_{e(v'')e(v')})$ is a matrix
%obtained by omitting $v''$th row $e(v'')$ and $v'$th column
%$e(v')$. By definition this is in fact
%$[(I-B^{\widetilde{G}_w})^{-1}]_{v'v''}=[P^{-1}]_{v'v''}$ and
%$\beta_{v'v''}^{\widetilde{G}_w}(x)$ is defined for those
%non-negative $x$ such that $x < \inf \{ y>0 \ \det (P(y))\neq 0
%\}$. For such an $x$, $C$ is also invertible as we mentioned
%earlier.
%
%The matrix $B^H$ is the reduced matrix of $-B^{\widetilde{G}_w}$.
Here $Row_v(N)$ is the row corresponding to vertex $v$.
%By letting $b=[x^{f(v_1)}F_{f,V,v_1} ,
%... ,x^{f(v_n)}F_{f,V,v_n} ]$,
%$c=[-\alpha_1F_{f,V,V_1},...,-\alpha_{\ell} F_{f,V,V_{\ell}}]$ and
Using Lemma \ref{redumat}, $\alpha_{H,f,w}(x)$ equals
 \begin{eqnarray*}
&& \sum_{V_i\subseteq \widetilde{V}_G\atop{ (w,V_i)\in
E(H)}}\langle
Row_{i}(C^{-1}),[-\alpha_1F_{f,V,V_1},...,-\alpha_{\ell}
 F_{f,V,V_{\ell}}]\rangle\\
&=& \sum_{v'\in \widetilde{V}_G\atop{ (w,v')\in E(G)}}\langle Row_{v'}(I-B^{\widetilde{G}_w})^{-1},[x^{f(v_1)}F_{f,V,v_1} ,... ,x^{f(v_n)}F_{f,V,v_n}]\rangle\\
&=& \sum_{v'\in V(\widetilde{G}_w)\cap \overline{A}_{G_w,f,w}}\sum_{v'' \in V(\widetilde{G}_w)}[(I-B^{\widetilde{G}_w})^{-1}]_{v'v''}x^{f(v'')}F_{f,V,v''}(x)\\
&=& \sum_{v'\in V(\widetilde{G}_w)\cap \overline{A}_{G_w,f,w}}\sum_{v'' \in V(\widetilde{G}_w)} \beta_{v'v''}^{\widetilde{G}_w}(x) x^{f(v'')}F_{f,V,v''}(x)\\
&=&  \alpha_{V(\widetilde{G}_w)\cap \overline{A}_{G_w,f,w}}(x).\\
\end{eqnarray*}
It is easy to see $\alpha_{V(\widetilde{G}_w),f,w}(x)$,
$\sigma_{G,f,w}(x)$ and $\phi_{\widetilde{G},f,w}(x)$ are equal
to $\alpha_{H,f,w}(x)$, $\sigma_{H,f,w}(x)$ and
$\widetilde{\phi}_{H,f,w}(x)$ respectively. Since
$\zeta+\alpha_{m}=F_{f,V}-\sum_{v\in \widetilde{V}_G}x^{f(v)}$,
the proof completes.
\end{proof}
\begin{rem}
TBS is a special case for TMC where then,
$$\phi_{G_0,f,v}(x)=\frac{x^{f(v)}}{1+x^{f(v)}-F_{f,V}(x)},$$ for
$1+x^{f(v)}-F_{f,V}(x)>0$ [3, Theorem 1].\hfill $\Box$
\end{rem}
\section{Criteria for the Existence of a Measure with Maximal
Entropy}\label{crit} This section is pretty short for there are
similar results in  \cite{P} which can be used here directly.
\begin{The}\label{maximal entropy}
Let ($Y(G),T$) be an RFT, $f\in \mathcal{F}^o(Y(G))$ and $T_f$ the
special flow constructed on $Y(G)$. The following statements are
equivalent:
\begin{itemize}
  \item[i)] $h(T_f)<\infty$ and $T_f$ has a (unique) measure with
  maximal entropy.
  \item[ii)] There exists $x_0>0$ such that $\phi_{G,f,w}(x_0)=1.$
\end{itemize}
\end{The}
\begin{proof}
By a result in \cite{S}, the existence of a measure with maximal
entropy is guaranteed if and only if  the following conditions are
satisfied.
\begin{itemize}
  \item [1)] $\sum_{\gamma \in
  C(G,w)}f^*(\gamma)e^{-h(T_f)f^*(\gamma)}< \infty$,
  \item [2)] $\sum_{\gamma \in C(G,w)}e^{-h(T_f)f^*(\gamma)}=1$.
\end{itemize}

%Condition (1)  means that $\phi_{G,f,w}(x)=1$ has a positive
%solution and this is  unique since $\phi_{G,f,w}(x)$ is
%increasing.
Condition (1) follows from the fact that $\phi_{G,f,w}(x)$ is
$C^1$ by Corollary (\ref{cor1}) and (2) means exactly that there
must be $x_0$ such that $\phi_{G,f,w}(x_0)=1$. It worths to
mention that $\phi_{G,f,w}(x)$ is an increasing function and if
ever $\phi_{G,f,w}(x_0)=1$, then $x_0$ must be unique.

%Then since $\phi_{G,f,w}(0)=0$ and $\phi_{G,f,w}(x)$ is
%increasing, there must be a unique $x_0>0$ where
%$\phi_{G,f,w}(x_0)=1$.
%
%To show that (2) is satisfied we only need to prove that
%$\phi_{G,f,w}(x)$ is continuously differentiable. This is the
%case if $M$ is invertible in the domain of $\phi_{G,f,w}(x)$ by
%Theorem \ref{cd}.
%
% Now suppose $\phi_{G,f,w}(x)$ is as in Theorem \ref{equiv}.
%The function $F_{f,V}(x)$ is continuously differentiable \cite{P}
%and by a similar argument $F_{f,V,V_i}(x), 1\leq i \leq \ell$ is
%continuously differentiable. Recall that  $C(x)=B^H(x) -Id$ and
%for $1\leq i,j\leq \ell$,
%$$\sum_{n=0}^{\infty}[B^H(x)]^n_{ij}=\frac{(-1)^{i+j}\det(C(x))_{e(j)e(i)}}{\det(C(x))}=(C)^{-1}_{ij}(x).$$
%Here $(C_{ij})^{-1}(x)$ is continuously differentiable for $x\in
%[0,x_0)$ where $x_0$ is the least positive real such that $\det
%(C(x_0))=0$. But the function $\phi_{G,f,w}(x)$ is a function
%with variables $F_{f,V}(x), F_{f,V,V_i}(x), x^{f(v)}$ and
%$(C)^{-1}(x)$. This implies $\phi_{G,f,w}(x)$ is continuously
%differentiable if
%$1-\zeta(x)-\alpha_m(x)-\alpha_{H,f,\widetilde{V}_G}(x)>0$.
\end{proof}
\begin{cor}
Suppose the hypothesis of Theorem \ref{maximal entropy}. Suppose
either
\begin{itemize}
  \item [1)] $\exists \, x>0\ \ni\ M(x)=0$, or
  \item [2)] $\lim_{x\rightarrow r(F_{f,V})^-} F_{f,V}(x)=\infty$.
\end{itemize}
Then the existence of a measure with  maximal entropy is
guaranteed.
\end{cor}
\begin{proof}
If (1) is satisfied, the proof is immediate from Theorem
\ref{det}. Now suppose (2) is satisfied, and recall that
$F_{f,V}(x)=\sum_{i=1}^{m}\alpha_{i}(x)=\sum_{i=1}^{\ell}\alpha_{i}(x)+\sum_{i=\ell
+1 }^{m}\alpha_{i}(x)$. So $\lim_{x\rightarrow
r(F_{f,V})^-}\sum_{i=1}^{\ell}\alpha_{i}(x)=\infty$ or
$\lim_{x\rightarrow r(F_{f,V})}\sum_{i=\ell}^{m}=\infty$. Then in
Theorem \ref{equv}, $\lim_{x\rightarrow r(F_{f,V})^-}(\zeta
+\alpha_{m})=\infty$ or $\lim_{x\rightarrow
r(F_{f,V})^-}\widetilde{\phi}_{H,f,w}(x)=\infty$ which either
implies the conclusion.
\end{proof}
Note that the first two examples in section \ref{egs} have
measures with maximal entropy and the last one dose not have such
a measure.
%Note that the first two
%************************
%\newpage

\bibliographystyle{amsplain}

\end{document}